\documentclass[11pt,english]{article}

\usepackage[margin=2 cm,bottom=15mm,footskip=7mm,top=15mm]{geometry}

\usepackage{amsmath}
\usepackage{amsthm}
\usepackage{amssymb}

\usepackage{setspace}
\usepackage{mathtools}
\usepackage{graphicx}
\usepackage[hidelinks]{hyperref}
\usepackage{thm-restate}
\usepackage{cleveref}
\usepackage{hyperref}
\usepackage{enumitem}
\usepackage{framed}
\usepackage{subcaption}

\usepackage[square,sort,comma,numbers]{natbib}
\usepackage{floatrow}

\usepackage[square,sort,comma,numbers]{natbib}
\theoremstyle{plain}

\newtheorem*{thm*}{Theorem}
\newtheorem{thm}{Theorem}[section]
\Crefname{thm}{Theorem}{Theorems}

\newtheorem*{lem*}{Lemma}
\newtheorem{lem}[thm]{Lemma}
\Crefname{lem}{Lemma}{Lemmas}

\newtheorem*{claim*}{Claim}

\crefname{claim}{Claim}{Claims}
\Crefname{claim}{Claim}{Claims}

\newtheorem{prop}[thm]{Proposition}
\Crefname{prop}{Proposition}{Propositions}

\crefname{cor}{Corollary}{Corollaries}

\newtheorem{conj}[thm]{Conjecture}
\crefname{conj}{Conjecture}{Conjectures}

\Crefname{qn}{Question}{Questions}

\newtheorem{obs}[thm]{Observation}
\Crefname{obs}{Observation}{Observations}

\Crefname{ex}{Example}{Examples}

\theoremstyle{definition}

\Crefname{prob}{Problem}{Problems}

\newtheorem{defn}[thm]{Definition}
\Crefname{defn}{Definition}{Definitions}

\theoremstyle{remark}

\renewenvironment{proof}[1][]{\begin{trivlist}
\item[\hspace{\labelsep}{\bf\noindent Proof#1.\/}] }{\qed\end{trivlist}}

\newcommand{\remove}[1]{}

\newcommand{\ceil}[1]{
    \lceil #1 \rceil
}
\newcommand{\floor}[1]{
    \lfloor #1 \rfloor
}

\newcommand{\eps}{\varepsilon}

\begin{document}

\title{\vspace{-0.9cm} $2$-factors with $k$ cycles in Hamiltonian graphs}

\author{
	Matija Buci\'c\thanks{
	    Department of Mathematics, 
	    ETH Zurich,
	    Switzerland;
	    e-mail: \texttt{matija.bucic}@\texttt{math.ethz.ch}.
	}
	\and
	Erik Jahn\thanks{
	    Department of Mathematics, 
	    ETH Zurich,
	    Switzerland;
	    e-mail: \texttt{ejahn}@\texttt{student.ethz.ch}.
	}
	\and
        Alexey Pokrovskiy\thanks{
        Department of Economics, Mathematics, and Statistics, Birkbeck,
        UK;
        e-mail: \texttt{dr.alexey.pokrovskiy}@\texttt{gmail.com}.
    }
    \and
	Benny Sudakov\thanks{
	    Department of Mathematics, 
	    ETH Zurich,
	    Switzerland;
	    e-mail: \texttt{benjamin.sudakov}@\texttt{math.ethz.ch}.     
		Research supported in part by SNSF grant 200021-175573.
	}
}
\date{}

\maketitle

\begin{abstract}
    \setlength{\parskip}{\smallskipamount}
    \setlength{\parindent}{0pt}
    \noindent
    
    \vspace{-0.9cm} 
A well known generalisation of Dirac's theorem states that if a graph $G$ on $n\ge 4k$ vertices has minimum degree at least $n/2$ then $G$ contains a $2$-factor consisting of exactly $k$ cycles. This is easily seen to be tight in terms of the bound on the minimum degree. However, if one assumes in addition that $G$ is Hamiltonian it has been conjectured that the bound on the minimum degree may be relaxed. This was indeed shown to be true by S\'ark\"ozy. In subsequent papers, the  minimum degree bound has been improved, most recently to $(2/5+\varepsilon)n$ by DeBiasio, Ferrara, and Morris. On the other hand no lower bounds close to this are known, and all papers on this topic ask whether the minimum degree needs to be linear. We answer this question, by showing that the required minimum degree for large Hamiltonian graphs to have a $2$-factor consisting of a fixed number of cycles is sublinear in $n.$ 
\end{abstract}

\section{Introduction}

A celebrated theorem by Dirac \cite{Dirac52} asserts the existence of a Hamilton cycle whenever the minimum degree of a graph $G$, denoted $\delta(G)$, is at least $\frac{n}{2}$. Moreover, this is best possible as can be seen from the complete bipartite graph $K_{\floor{\frac{n-1}{2}},\ceil{\frac{n+1}{2}}}$. Dirac's theorem is one of the most influential results in the study of Hamiltonicity of graphs and has seen generalisations in many directions over the years (for some examples consider surveys \cite{lisurvey,gouldsurvey,bennysurvey} and references therein). In this paper we discuss one such direction by considering what conditions ensure that we can find various \textit{$2$-factors} in $G$. Here, a \textit{$2$-factor} is a spanning 2-regular subgraph of $G$ or equivalently, a union of vertex-disjoint cycles that contains every vertex of $G$ and hence, $2$-factors can be seen as a natural generalisation of Hamilton cycles. Brandt, Chen, Faudree, Gould and Lesniak \cite{brandt97} proved that for a large enough graph the same degree condition as in Dirac's theorem, $\delta(G)\ge n/2$, allows one to find a $2$-factor with exactly $k$ cycles. 

\begin{thm}
If $k \geq 1$ is an integer and $G$ is a graph of order $n \geq 4k$ such that $\delta(G) \geq \frac{n}{2}$, then $G$ has a $2$-factor consisting of exactly $k$ cycles.
\end{thm}

Once again, this theorem gives the best possible bound on the minimum degree, using the same example as for the tightness of Dirac's theorem above. This indicates that perhaps if we restrict our attention to Hamiltonian graphs, thereby excluding this example, a smaller minimum degree might be enough. That this is in fact the case was conjectured by Faudree, Gould, Jacobson, Lesniak and Saito \cite{Faudree05}.

\begin{conj}\label{conj:main}
For any $k \in \mathbb{N}$ there are constants $c_k <1/2,$ $n_k$ and $a_k$ such that any Hamiltonian graph $G$ of order $n\ge n_k$ with $\delta(G) \ge c_k n+a_k$ contains a $2$-factor consisting of $k$ cycles.
\end{conj}

Faudree et al.\ prove their conjecture for $k=2$ with $c_2=5/12.$ 

The conjecture was shown to be true for all $k$ by S\'{a}rk\"{o}zy \cite{Sarkozy08} with $c_k=1/2-\eps$ for an uncomputed small value of $\eps>0.$ Gy\"ori and Li \cite{gyori12} announced that they can show that $c_k=5/11+\eps$ suffices. The best known bound was due to DeBiasio, Ferrara and Morris \cite{DeBiasio14} who show that  $c_k= \frac{2}{5}+\eps$ suffices.

On the other hand no constructions of very high degree Hamiltonian graphs without $2$-factors of $k$ cycles are known. Faudree et al.~ \cite{Faudree05} say ``we do not know whether a linear bound of minimum degree in Conjecture~\ref{conj:main} is appropriate''. Sark\"ozy~\cite{Sarkozy08} says ``the obtained bound on the minimum degree is probably far from
best possible; in fact, the ``right'' bound might not even be linear''. DeBiasio et al.~\cite{DeBiasio14} say ``one vexing aspect of Conjecture~\ref{conj:main} and the related work described here is that it is possible that a sublinear, or even constant, minimum degree would suffice to ensure a Hamiltonian graph has a 2-factor of the desired type''. In particular, in \cite{Faudree05,Sarkozy08,DeBiasio14} they all ask the question of whether the minimum degree needs to be linear in order to guarantee a $2$-factor consisting of $k$ cycles. We answer this question by showing that the minimum degree required to find $2$-factors consisting of $k$ cycles in Hamiltonian graphs is indeed sublinear in $n.$

\begin{thm}\label{thm:main}
For every $k\in \mathbb{N}$ and $ \eps>0$, there exists $N = N(k,\eps)$ such that if $G$ is a Hamiltonian graph on $n \geq N$ vertices with $\delta (G) \geq \eps n$, then $G$ has a $2$-factor consisting of $k$ cycles.
\end{thm}

\subsection{An overview of the proof}
We now give an overview of the proof to help the reader navigate the rest of the paper.

In the next section we will show that any $2$-edge-coloured graph $G$ on $n$ vertices with minimum degree being linear in both colours contains a blow-up of a short colour-alternating cycle. This is an auxiliary result which we need for our main proof. There, we also introduce ordered graphs and show a result which, given an ordering of the vertices of $G$ allows us to find a blow-up as above that is also consistent with the ordering, meaning that given two parts of the blow-up, vertices of one part all come before the other.

The main part of the proof appears in \Cref{sec:main}. The key idea is given a graph $G$ with a Hamilton cycle $H=v_1\ldots v_nv_1$, to build an auxiliary $2$-edge-coloured graph $A$ whose vertex set is the set of edges $e_i=v_iv_{i+1}$ of $H$ and for any edge $v_iv_j\in G\setminus H$ we have a red edge between $e_i$ and $e_j$ and a blue edge between $e_{i-1}$ and $e_{j-1}$ in $A$.
The crucial property of $A$ is that given any vertex disjoint union of colour-alternating cycles $S$ in $A$ one can find a $2$-factor $F(S)$ in $G$, consisting of the edges of $H$ which are not vertices of $S$ and the edges of $G$ not in $H$ which gave rise to the edges of $S$ in $A$. 

However, we can not control the number of cycles in $F(S)$ (except knowing that $F(S)$ has at most $|S|$ cycles), since it depends on the structure of $S$ and also on how $S$ is embedded within $A$. To circumvent this issue we will find instead a large blow-up of $S$. Then within this blow-up we show how to find a modification of $S$ denoted $S^+$ which has the property that $F(S^+)$ has precisely one cycle more than $F(S)$. Similarly, we find another modification $S^-$ such that the corresponding $2$-factor $F(S^-)$ has precisely one cycle less than $F(S)$. Since the number of cycles in $F(S)$ is bounded, if our blow-up of $S$ is sufficiently large we can perform these operations multiple times and therefore obtain a $2$-factor with the target number of cycles.

\section{Preliminaries}\label{sec:prelim}
Let us first fix some notation and conventions that we use throughout the paper. For a graph $G=(V,E)$, let $\delta(G)$ denote its minimum degree, $\Delta(G)$ its maximum degree and $d(v)$ the degree of a vertex $v \in V$. For us, a \emph{$2$-edge-coloured graph} is a triple $G = (V, E_1, E_2)$ such that both $G_1 = (V, E_1)$ and $G_2 = (V,E_2)$ are simple graphs. We always think of $E_1$ as the set of \emph{red} edges and of $E_2$ as the set of \emph{blue} edges of $G$. Accordingly, we define $\delta_1(G)$ to be the minimum degree of red edges of $G$ (that is $\delta(G_1)$), and analogously $\Delta_1(G)$, $\delta_2(G)$, etc. Note that with our definition the same two vertices may be connected by two edges with different colours. In this case, we say that $G$ has a \emph{double edge}. A \textit{blow-up} $G(t)$ of a $2$-edge coloured graph $G$ (with no two vertices joined by both a red and a blue edge) is constructed by replacing each vertex $v$ with a set of $t$ independent vertices and adding a complete bipartite graph between any two such sets corresponding to adjacent vertices in the colour of their edge. When working with \emph{digraphs} we always assume they are simple, so without loops and with at most one edge from any vertex to another (but we allow edges in both directions between the same two vertices).

\subsection{Colour-alternating cycles}\label{subs:cycles}
In this subsection, our goal is to prove that any $2$-edge-coloured graph, which is dense in both colours contains a blow-up of a colour-alternating cycle. We begin with the following auxiliary lemma that will only be used in the subsequent lemma where we will apply it to a suitable auxiliary digraph to give rise to many colour-alternating cycles.
\begin{lem}
Let $k \ge 2$ be a positive integer. A directed graph on $n$ vertices with minimum out-degree at least $\frac{n\log (2k)}{k-1}$ has at least $\frac{n^\ell}{2k^{\ell+1}}$ cycles of length $\ell$ for some fixed $2\le \ell\le k.$
\end{lem}

\begin{proof}
Let us sample $k$ vertices $v_1,\ldots, v_{k}$ from $V(G),$ independently, uniformly at random, with repetition. We denote by $X_i$ the event that vertex $v_i$ has no out-neighbour in $S:=\{v_1,\ldots, v_k\}.$ We know that $\mathbb{P}(X_i)\le \left(1-\frac{\log (2k)}{k-1}\right)^{k-1}\le \frac{1}{2k}.$ If no $X_i$ occurs then the subgraph induced by $S$ has minimum out-degree at least $1$ so contains a directed cycle. The probability of this occurring is at least:
$$ \mathbb{P}\left(\overline{X_1}\cap \ldots \cap \overline{X_k}\right)=1-\mathbb{P}(X_1\cup \ldots \cup X_k)\ge 1-k\mathbb{P}(X_i) \ge 1/2,$$
where we used the union bound. This means that in at least $n^k/2$ outcomes we can find a cycle of length at most $k$ within $S.$ In particular, there is an $\ell \le k$ such that in at least $\frac{n^k}{2k}$ outcomes the cycle we find has length exactly $\ell$. Note that the same cycle might have been counted multiple times, but at most $k^\ell n^{k-\ell}$ times. This implies that $C_\ell$ occurs at least $\frac{n^\ell}{2k^{\ell+1}}$ times.  
\end{proof}

Now, we use this lemma to conclude that there are many copies of some short colour-alternating cycle in any $2$-edge-coloured graph which has big minimum degree in both colours.

\begin{lem}\label{lem:many-cycles}
For every $\gamma \in (0,1)$ there exist $c=c(\gamma), L = L(\gamma)$ and  $K = K(\gamma)$ such that, if $G$ is a $2$-edge-coloured graph on $n \geq K$ vertices satisfying $\delta_1(G), \delta_2(G) \geq \gamma n$, then $G$ contains at least $cn^\ell$ copies of a colour-alternating cycle of some fixed length $4 \le \ell \le L$.
\end{lem}
\begin{proof}
Let $k=8/\gamma^2 \log (8/\gamma^2)$ so that $\gamma^2/4 \ge \log (2k)/(k-1).$ We set $L=2k,$ $K=8k/\gamma^2$ and $c=(\gamma/2)^{2k}/(4k^{k+1}).$ 
We build a digraph $D$ on the same vertex set as $G$ by placing an edge from $v$ to $u$ if and only if there are at least $\gamma^2n/2$ vertices $w$ such that $vw$ is red and $wu$ is blue. 

Let us first show that every vertex of $D$ has out-degree at least $\gamma^2n/4.$ There are at least $\gamma n$ red neighbours of $v$ and each has $\gamma n$ blue neighbours so there are at least $\gamma^2n^2$ red-blue paths of length $2$ starting at $v.$ Let us assume that there are less than $\gamma^2n/2$ vertices $u$ such that there are at least $\gamma^2n/2$ vertices $w$ such that $vw$ is red and $wu$ is blue. 
In this case there are less than $\gamma^2n/2 \cdot n+ n \cdot \gamma^2n/2$ red-blue paths starting at $v$ which is a contradiction. Note that we allowed $u=v$ in the above consideration so we deduce that minimum out-degree in $D$ is at least $\gamma^2n/2-1\ge \gamma^2n/4.$ 
The previous lemma implies that there is some $\ell \le k$ such that $D$ contains at least $n^{\ell}/(2k^{\ell+1})$ copies of $C_\ell.$ 

For any such cycle by replacing each directed edge by a red-blue path of $G$ between its endpoints, ensuring we don't reuse a vertex, we obtain at least $(\gamma^2n/2-\ell)(\gamma^2n/2-\ell-1)\cdots(\gamma^2n/2-2\ell+1)\ge (\gamma/2)^{2\ell}n^\ell$ 
colour-alternating $C_{2\ell}$'s in $G$. Noticing that each such $C_{2\ell}$ may arise in at most $2$ different ways from a directed $C_{\ell}$ of $D$ we deduce that there are at least $n^{\ell}/(2k^{\ell+1})\cdot(\gamma/2)^{2\ell}n^\ell/2\ge c(\gamma)n^{2\ell}$ colour-alternating $C_{2\ell}$'s in $G$.
\end{proof}

The reason for formulating the above lemma is that we can deduce the existence of the blow-up of a cycle from the existence of many copies of this cycle using the hypergraph version of the celebrated K\H{o}v\'ari-S\'os-Tur\'an theorem proved by Erd\H{o}s in \cite{kst}:

\begin{thm}\label{thm:kst}
Let $\ell,t \in \mathbb{N}$. There exists $C=C(\ell,t)$ such that any $\ell$-graph on $n$ vertices with at least $Cn^{\ell-1/t^\ell}$ edges contains $K^{(\ell)}(t)$, the complete $\ell$-partite hypergraph with parts of size $t,$ as a subgraph. 
\end{thm}

We are now ready to find our desired blow-up.
\begin{lem}\label{lem:cycle-blow-up}
For every $\gamma \in (0,1)$ and $t\in \mathbb{N}$, there exist positive integers $L = L(\gamma)$ and  $K = K(\gamma,t)$ such that, if $G$ is a $2$-edge-coloured graph on $n \geq K$ vertices satisfying $\delta_1(G), \delta_2(G) \geq \gamma n$, then $G$ contains $\mathcal{C}(t6^L)$ where $\mathcal{C}$ is a colour-alternating cycle with $|V(\mathcal{C})| \leq L.$
\end{lem}
\begin{proof}
Let $L=L(\gamma),c=c(\gamma),K \ge K(\gamma)$ be parameters of \Cref{lem:many-cycles} so that we can find $cn^\ell$ copies of a colour-alternating cycle of length $4 \le \ell \le L.$ Let $C=C(L,t6^L)\ge C(\ell,t6^L)$ be the parameter given by \Cref{thm:kst}. By assigning each vertex of $V(G)$ into one of $\ell$ parts uniformly at random we can find a partition of $V(G)$ into $V_1,\ldots,V_\ell$ such that there are $cn^\ell/\ell^\ell$ colour-alternating cycles $v_1\ldots v_\ell$ with $v_i \in V_i$. 
We also know that at least half of these cycles always use edges of the same colour between all $V_i,V_{i+1}.$ We now build an $\ell$-graph $H$ on the same vertex set as $G$ whose edges correspond to sets of vertices of such colour-alternating cycles. So we know $H$ has at least $\frac{c}{2\ell^\ell}n^\ell  \ge Cn^{\ell-1/(t^\ell\cdot 6^{\ell L})}$ many edges, by taking $K$ large enough, depending on $t,L.$ So \Cref{thm:kst} implies that $H$ contains $K^{(\ell)}(t6^L)$ as a subgraph, which corresponds to a desired $\mathcal{C}(t6^L).$
\end{proof}

\subsection{Ordered graphs}
In our arguments it will not be enough to just find a blow-up of a colour-alternating cycle as in the previous subsection; we will also care about the ``order'' in which the cycles are embedded. In this section we give some notation about ordered graphs and a result which we will need later.

An \emph{ordered graph} is a graph together with a total order of its vertex set.
Here, whenever $G$ is a graph on an indexed vertex set $V(G) = \{v_1, \dots, v_n\}$, we assume that $G$ is ordered by $v_i < v_j \iff i < j$. An \emph{ordered subgraph} of an ordered graph $G$ is a subgraph of $G$ that is endowed with the order that is induced by $G$ and if not stated otherwise, we assume that subgraphs of $G$ are always endowed with that order. For us, two vertices $u < v$ of an ordered graph $G$ are called \emph{neighbouring}, if the set of vertices between $u$ and $v$, that is $\{x \in V(G) | u \leq x \leq v\}$, is either just $\{u,v\}$ or the whole vertex set $V(G)$.

Given an ordered graph $G$ we say a blow-up $H=G(k)$ of $G$ is \textit{ordered consistently} if for any $x,y \in V(H)$ which belong to parts of the blow-up coming from vertices $u,v\in G$ respectively we have $x <_H y$ iff $u <_G v.$

\begin{lem}\label{lem:ordered}
Let $t,L\in \mathbb{N},$ $H$ be a graph on $L$ vertices and $H(t2^L)\subseteq G$ for an ordered graph $G$. There exists an ordering of $H$ for which the consistently ordered $H(t)$ is an ordered subgraph of $G.$ 
\end{lem}
\begin{proof}
We prove the result by induction on $L,$ where the $L=1$ case is immediate. Let $\{V_1, \dots, V_L\}$ be the clusters of vertices of $H(t2^L),$ so $|V_i|=t2^L.$ Let $w_1, \dots, w_p$ be the median vertices of the sets $V_1, \dots ,V_p$ with respect to the ordering of $H(t2^L)$ induced by $G$ and assume without loss of generality that $w_1$ is the smallest of them. We now throw away all vertices of $V_1$ that are larger than $w_1$ and all vertices of $V_i$ that are smaller than $w_i$ for $i \geq 2$. This leaves us with $L$ sets $\{W_1, \dots, W_L\}$ of size $\ceil{|V_i|/2}=t2^{L-1}$ with the property that $v_1 \in W_1, v_i \in W_i \implies v_1 <_G w_1<_G w_i<_G v_i$ for all $i \geq 2$. If $v\in H$ corresponds to $V_1$ and we denote $H'=H-v$ then $\mathcal{W} = \{W_2, \dots, W_L\}$ spans $H'(t2^{L-1})\subseteq G\setminus V_1$. By the induction hypothesis we can find a consistently ordered $H'(t)$ as an ordered subgraph of $G\setminus V_1$ which together with any subset of size $t$ of $W_1$ gives the desired consistently ordered $H(t)$ in $G$.
\end{proof}

\section{Proof of \Cref{thm:main}}\label{sec:main}

\subsection{Constructing an auxiliary graph}\label{constrA}

Throughout the whole section, let $G$ be a Hamiltonian graph on $n$ vertices. First of all, let us fix a Hamilton cycle $H$ of $G$ and name the vertices of $G$ such that $H = v_1 v_2 \dots v_n v_1$. We assume that $G$ is ordered according to this labelling. Also, let us denote the edges of $H$ by $e_1, e_2, \dots , e_n$ such that $e_1 = v_1 v_2, \dots , e_n = v_n v_1$. In all our following statements, we will identify $v_{n+1}$ and $v_1$, and more generally $v_i$ and $v_j$, as well as $e_i$ and $e_j$, if $i$ and $j$ are congruent modulo $n$. Furthermore, since we can always picture $G$ as a large cycle with some edges inside it, we call all the edges that are not part of $H$, the \emph{inner edges} of $G$.

Our goal is to find a $2$-factor with a fixed number of cycles in $G$. Note that, if $G$ is dense, it is not hard to find a large collection of vertex-disjoint cycles in $G$. The difficulty lies in the fact that we want this collection to be spanning while still controlling the exact number of cycles. Naturally, we have to rely on the Hamiltonian structure of $G$ to give us such a spanning collection of cycles. Indeed, when building these cycles we will try to use large parts of the Hamilton cycle $H$ as a whole and connect them correctly using some inner edges of $G$. It is convenient for our approach to construct an auxiliary graph $A$ out of $G$, that captures the information we need about the inner edges of $G$. 

\begin{defn} \label{defn:aux}

Given the setup above, we define the \emph{auxiliary graph} $A = A(G,H)$ as the following ordered, $2$-edge-coloured $n$-vertex graph:

\begin{enumerate}
    \item Every vertex of $A$ corresponds to exactly one edge of $H$, thus we have $V(A) = \{e_1, \dots , e_n\}$ and we order the vertices of $A$ according to this labelling;
    \item two vertices $e_i = v_i v_{i+1}$ and $e_j = v_j v_{j+1}$ of $A$ are connected with a red edge if there is an inner edge of $G$ connecting $v_{i+1}$ and $v_{j+1}$;
    \item similarly, the vertices $e_i$ and $e_j$ of $A$ are connected with a blue edge if there is an inner edge of $G$ connecting $v_i$ and $v_j$.
\end{enumerate}

\end{defn}

 Throughout this section, let $A = A(G,H)$ for our fixed $G$ and $H$. Note that, by the above definition, every edge $\ell \in E(A)$ corresponds to a unique inner edge $e$ of $G$. In the following, we denote this edge by $e(\ell) \in E(G)$. To be precise, if $\ell = e_i e_j$, then $e(\ell) \coloneqq v_{i+1} v_{j+1}$ if $\ell$ is a red edge and $e(\ell) \coloneqq v_i v_j$ if $\ell$ is a blue edge. Conversely, every inner edge of $G$ corresponds to exactly one red edge and to one blue edge of $A$. This leads to the following observation:

\begin{obs} \label{deltaA}
For $i \in \{1, \dots, n\}$, we have $d^{A}_1(e_i) = d^{G}(v_{i+1})-2$ and $d^{A}_2(e_i) = d^{G}(v_i)-2$. In particular, we have $\delta_1(A) = \delta_2(A) = \delta(G) - 2$.
\end{obs}

In \Cref{ExA} we give an example of a Hamiltonian graph and its corresponding auxiliary graph.

\begin{figure}[ht]
    \caption{Let us call the left graph $G$ and fix its Hamilton cycle $H = v_1 \dots v_8 v_1$. Then the graph on the right is the auxiliary graph $A(G,H)$.}
    \includegraphics[scale = 0.8]{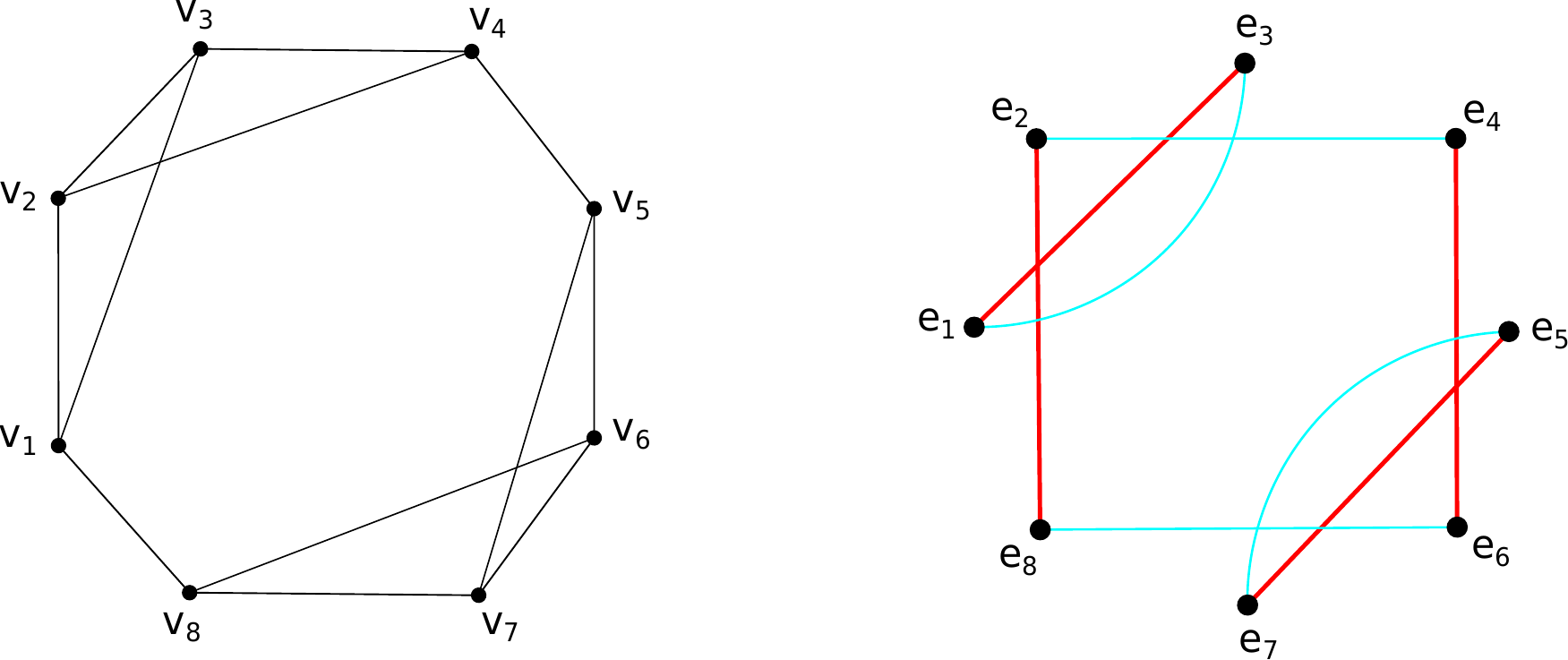}
    \label{ExA}
\end{figure}

The motivation for defining $A$ just as above is given by the fact that 2-regular (possibly non-spanning!) subgraphs $S \subseteq A$ satisfying some extra conditions naturally correspond to a $2$-factor in $G$. Recall that in our setting, two vertices $e_i$ and $e_j$ of $A$ are neighbouring if $|i-j| \equiv 1$ (modulo $n$). Let us make the following definition:

\begin{defn} \label{defn:correspondence}
Given the same setup as above and a subgraph $S \subseteq A$ that is a union of vertex-disjoint colour-alternating cycles without neighbouring vertices (i.e.\ if $e_i \in V(S)$ then $e_{i-1},e_{i+1}\notin V(S)$), we define its corresponding subgraph $F(S) \subseteq G$ as follows:

\begin{enumerate}
    \item $V(F(S)) \coloneqq V(G)$;
    \item the edges of $F(S)$ are all the edges of $H$ except for those that correspond to vertices of $S$. Additionally, for every edge $\ell \in E(S)$, let the corresponding inner edge $e(\ell)$ be an edge of $F(S)$ too. That is, $E(F(S)) \coloneqq \left(\{e_1, \ldots, e_n\} \setminus V(S)\right) \cup \{e(\ell) \mid \ell \in E(S)\}$.
\end{enumerate}
\end{defn}

\begin{lem} \label{correspondence}
If $S \subseteq A$ is a union of vertex-disjoint colour-alternating cycles without neighbouring vertices, then $F(S) \subseteq G$ is a $2$-factor. 
\end{lem}

In order to illustrate the above definitions, consider the Hamiltonian graph given in \Cref{ExA} and the subgraphs $S_1$ and $S_2$ of the corresponding auxiliary graph where $S_1$ is just the cycle $e_2 e_4 e_6 e_8 e_2$ and $S_2$ is the union of the cycles $e_1 e_3 e_1$ and $e_5 e_7 e_5$. Their corresponding $2$-factors $F(S_1)$ and $F(S_2)$ are shown as dashed in \Cref{ExCorr}. Note that they use the same inner edges of $G$ but still have different numbers of cycles.

\begin{figure}[ht]
    \caption{$2$-factors $F(S_1)$ and $F(S_2)$ used in the illustration above.}
    \includegraphics[scale = 0.8]{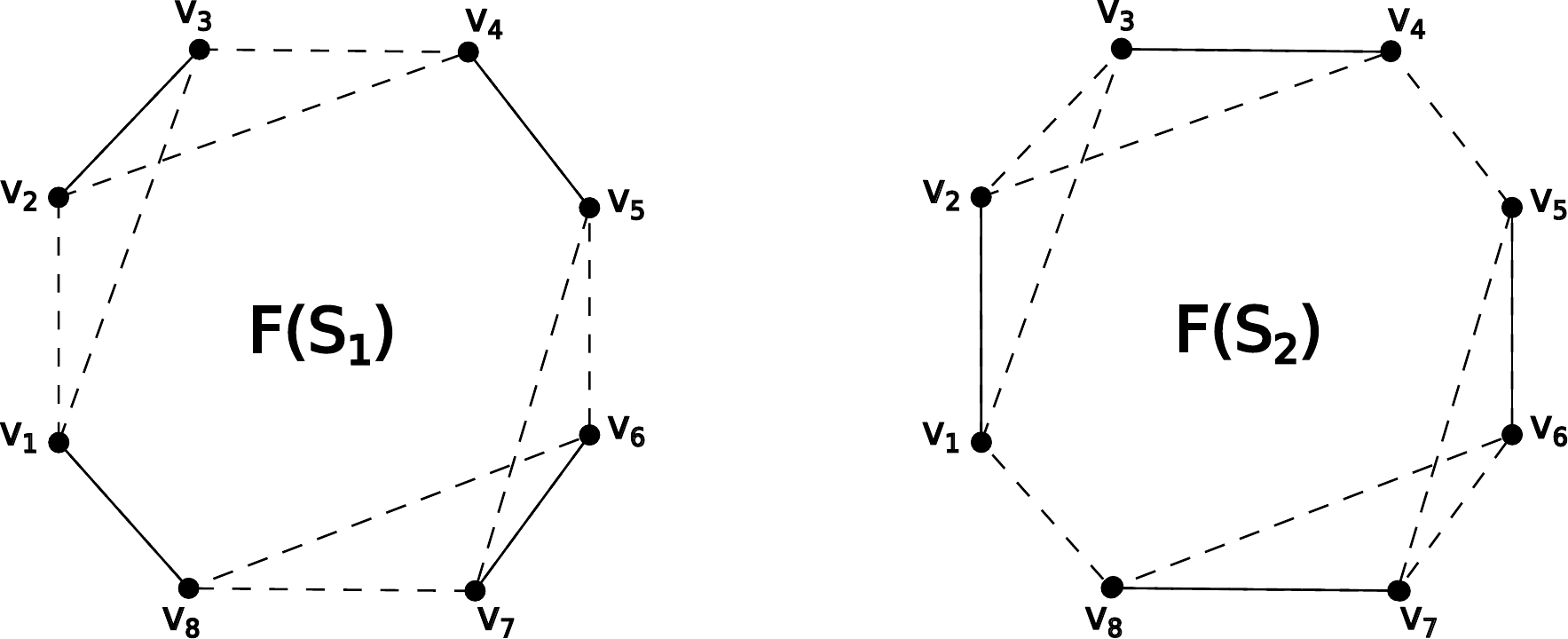}
    \label{ExCorr}
\end{figure}

\begin{proof}[ of \Cref{correspondence}]
Since $F \coloneqq F(S)$ consists of exactly $n$ edges, it suffices to show that $\delta(F)\ge 2$. Let $v_j$ be an arbitrary vertex of $F$. We distinguish two cases: If both edges $e_{j-1}, e_j \notin V(S)$, then $e_{j-1}, e_j \in E(F)$ and $v_j$ is incident to $e_{j-1}$ and $e_j$ in $F$. Else, exactly one of the edges $e_{j-1}$ and $e_j$ is a vertex of $S$ since $S$ contains no neighbouring vertices. In this case we use the fact that every vertex $e_i$ of $S$ is incident to a red edge $\ell_i$ and to a blue edge $\ell_i'$. Hence, by \Cref{defn:correspondence}, either $e_{j-1}\in S$ and $e_j \notin S$ in which case $v_j$ is incident to $e_j$ and $e(\ell_{j-1})$ in $F$ or $e_{j-1}\notin S$ and $e_j \in S$ in which case $v_j$ is incident to $e_{j-1}$ and $e(\ell_j')$ in $F$. In both cases these two edges are distinct as one of them is an inner edge of $G$ and the other one is not. 
\end{proof}

We note that $F(S)$ does not only depend on the structure of $S$ but also on the order in which $S$ is embedded within $A$. However, it is immediate that if $S$ is embedded in auxiliary graphs of two Hamiltonian graphs (possibly with different number of vertices) in the same order then $F(S)$ has the same number of cycles in both cases. 

\begin{obs}\label{obs:samecomps}
Let $A_1=A(G_1,H_1)$ and $A_2=A(G_2,H_2)$. Let $S_1$ and $S_2$ be disjoint unions of colour-alternating cycles without neighbouring vertices, which are isomorphic as coloured subgraphs of $A_1$ and $A_2$ whose corresponding vertices appear in the same order along $H_1$ and $H_2.$ Then $F(S_1)$ and $F(S_2)$ consist of the same number of cycles.
\end{obs}

We remark that it is not always true that all $2$-factors of $G$ arise as $F(S)$ for some $S \subseteq A.$

\subsection{Controlling the number of cycles} \label{control}

It is not hard to see that the auxiliary graph $A$ (of a graph with a big enough minimum degree) must contain a colour-alternating cycle $C$, which corresponds to a $2$-factor $F(C) \subseteq G$ by \Cref{correspondence} (disregarding, for the moment, the issue of $C$ containing neighbouring vertices). However, it is not at all obvious how to generally determine the number of components of $F(C).$ 
We begin by giving a rough upper bound.

\begin{obs} \label{obs:no.cycles}
If $C \subseteq A$ is a non-empty colour-alternating cycle of length $L$ without neighbouring vertices, then the number of components of the corresponding $2$-factor $F(C)$ is at most $L$.
\end{obs}

\begin{proof}
Note that the $2$-factor $F(C)$ contains exactly $L$ inner edges and, since $F(C) \neq H$, each cycle of $F(C)$ must contain at least one inner edge (in fact, at least two in our setting).
\end{proof} 

However, in order to prove \Cref{thm:main}, we need to be able to show the existence of a $2$-factor consisting of exactly $k$ cycles, for a fixed predetermined number $k$. This is where we are going to make use of \Cref{lem:cycle-blow-up,lem:ordered} which allow us to find a consistently ordered blow-up of $C$.  This will give us the freedom to find slight modifications of $C$ with different numbers of cycles in $F(C)$.

\subsubsection{Going up}
In this subsection we give a modification of a union of colour-alternating cycles which will have precisely one more cycle in its corresponding $2$-factor. 
\begin{defn}\label{defn:goingup}
Let $S$ be a disjoint union of colour-alternating cycles with $V(S)=\{s_1, \dots, s_m\}$ and let $C$ be a cycle of $S$. We construct a $2$-edge-coloured ordered graph $U(S,C)$ as follows:
\begin{enumerate}
    \item Start with a copy of $S$ and for every $s_i \in V(C)$, add a vertex $s_{i+1/2}$;
    \item For every red or blue edge $s_i s_j \in E(C)$, add an edge $s_{i+1/2} s_{j+1/2}$ of the same colour;
    \item Order the resulting graph according to the order of the indices of its vertices.
\end{enumerate}
Given a $2$-edge-coloured ordered graph $U$, we say that $U$ is a \emph{going-up version} of $S$, if there exists a component $C$ of $S$ such that $U$ and $U(S,C)$ are isomorphic $2$-edge-coloured ordered graphs.
\end{defn}

In other words $U(S,C)$ consists of $S$ with an additional copy of $C$ ordered in such a way that the vertices of the new copy of $C$ immediately follow their corresponding vertices of the original copy of $C$. In particular, $U$ is also a disjoint union of colour-alternating cycles and is an ordered subgraph of a consistently ordered $S(2)$. Note if $S$ contains no double edges, neither does $U.$

\Cref{fig:goingup} shows what a going-up version $U$ of $S$ looks like if $S$ is just a colour-alternating $C_4$. \Cref{fig:goingupG} shows what the corresponding $2$-factors look like (assuming $S \subseteq U \subseteq A$). Note that the dashed cycles of $F(U)$ have the same structure as the dashed cycles in $F(S)$ but $F(U)$ additionally has a new bold cycle. We now show that a similar situation occurs in general.

\begin{figure}[ht]
    \caption{A colour-alternating cycle $S$ and a going-up version of it $U$}
    \includegraphics[scale = 0.8]{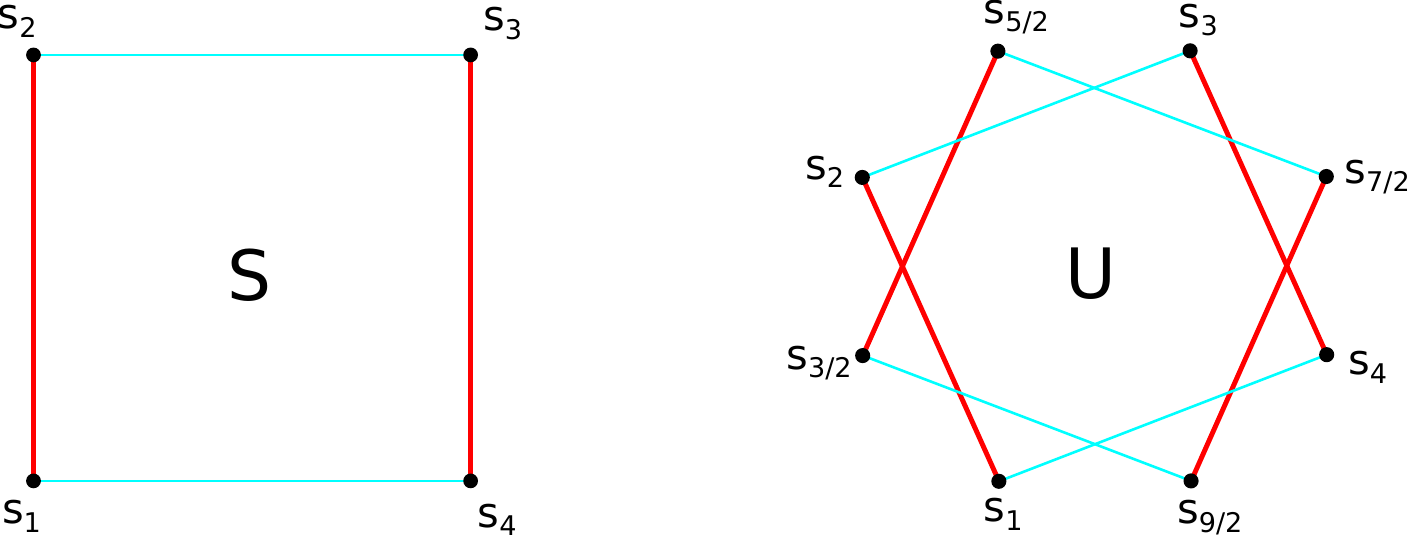}
    \label{fig:goingup}
\end{figure}

\begin{figure}[ht]
    \caption{$2$-factors corresponding to $U$ and $S$ given in \Cref{fig:goingup}.}
    \includegraphics[scale = 0.8]{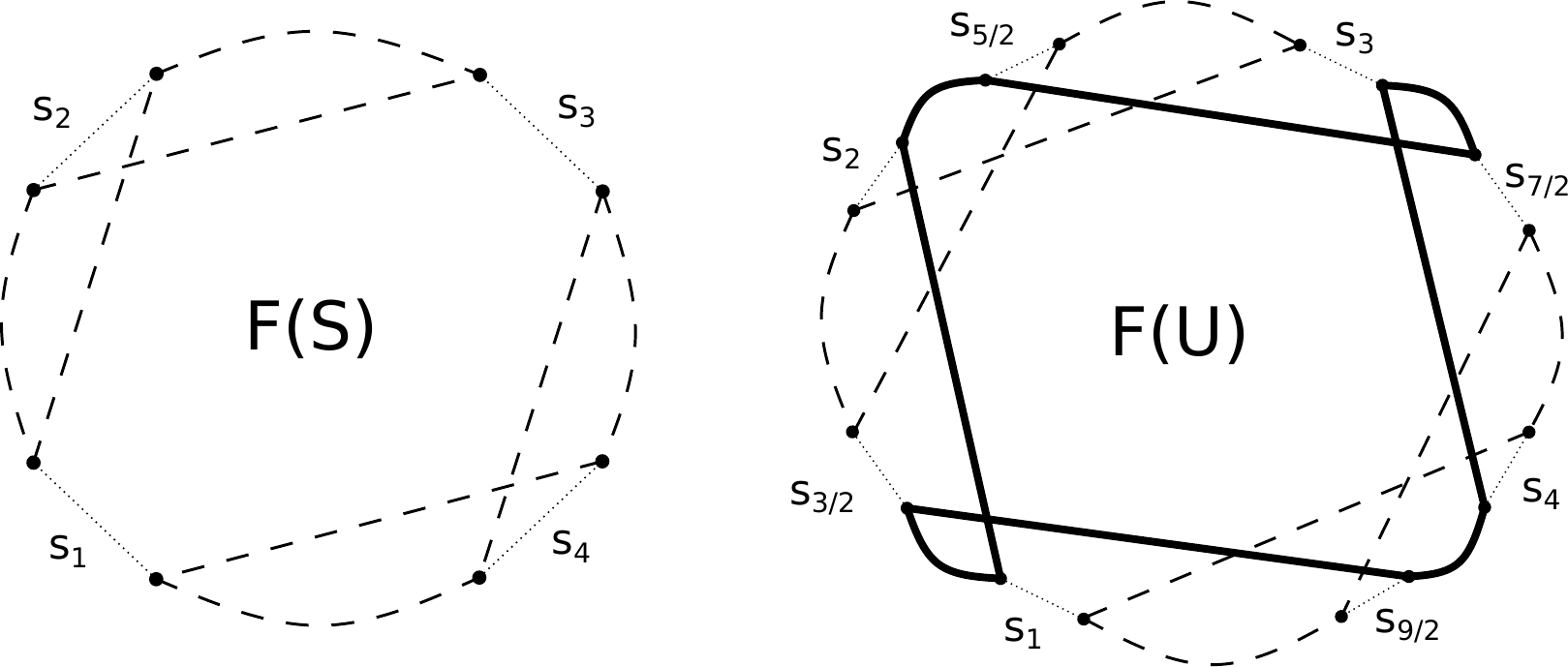}
    \label{fig:goingupG}
\end{figure}

\begin{lem}[Going up]\label{lem:goingup}
Let $S \subseteq A$ be a disjoint union of colour-alternating cycles without neighbouring vertices and let $U$ be an ordered subgraph of $A$ without neighbouring vertices that is a going-up version of $S$. Then, the $2$-factor $F(U) \subseteq G$ has exactly one component more than $F(S)$.
\end{lem}

\begin{proof}[ of \Cref{lem:goingup}]
For an edge $e=v_k v_{k+1}\in H$ we let $v^{+}(e)=v_{k+1}$ and $v^{-}(e)=v_k.$ 
We denote the vertices of $S$ by $s_1, \dots, s_m$ according to their order in $A$. Let $C$ be a colour-alternating cycle $s_{j_1} \dots s_{j_k} s_{j_1}$ in $S$ for which $U=U(S,C)$. Let us denote the vertices of $U$ by $u_1, \dots, u_m$ and $u_{j_1 + 1/2}, \dots, u_{j_k+1/2}$ as they appear along $H$ such that $u_1, \dots, u_m$ make a copy of $S$ and $u_{j_1}, \dots, u_{j_k}$ correspond to $C$.
The vertices $v^+(u_{j_i})$ and $v^-(u_{j_i + 1/2})$ are connected in $F(U)$ by paths $P_i \subseteq H$ for $i \in \{1, \dots, k\}$. Furthermore, since $C$ is a colour-alternating cycle either $v^+(u_{j_i})v^+(u_{j_{i+1}})\in E(G)$ for all odd $i$ and $v^-(u_{j_i + 1/2})v^-(u_{j_{i+1} + 1/2})\in E(G)$ for all even $i$ or vice versa in terms of parity. This means that taking all $P_i$ and these edges we obtain one cycle $$Z:=v^+(u_{j_1})v^+(u_{j_2})P_2v^-(u_{j_2+1/2})v^-(u_{j_3+1/2})P_3v^{+}(u_{j_3})\dots P_kv^-(u_{j_k + 1/2})v^-(u_{j_1+1/2})P_1v^+(u_{j_1})\in F(U),$$ if $C$ starts with a red edge (which is exactly the bold cycle in the example shown in \Cref{fig:goingupG}) or 
$$Z:=v^-(u_{j_1+1/2})v^-(u_{j_2+1/2})P_2v^+(u_{j_2})v^+(u_{j_3})P_3v^{-}(u_{j_3+1/2})\dots P_kv^+(u_{j_k})v^+(u_{j_1})P_1v^-(u_{j_1+1/2})\in F(U),$$
if $C$ starts with a blue edge.

Let us now consider the graph $G'$ that is obtained from $G$ by deleting $Z$ (including all edges incident to vertices of $Z$) and adding the edges $S_{j_i} = v^-(u_{j_i}) v^+(u_{j_i + 1/2})$ for $i \in \{1, \dots, k\}$. Let $H'$ be the Hamilton cycle of $G'$ made of $H$ and $S_j$'s ordered according to the order of $G$. We claim that sending the vertices $s_i$ to $S_i$ if $s_i \in C$ and to $u_i$ otherwise for $i \in \{1, \dots, m\}$ gives an order-preserving isomorphism from $S$ to its image $S' \subseteq A(G', H')$. Indeed, if $s_i, s_j \notin C$, then the fact that $u_i u_j$ is a red or a blue edge whenever $s_i s_j$ is a red or a blue edge just follows from \Cref{defn:goingup}. Furthermore, if $s_{j_i} s_{j_{i+1}}$ is a red edge for $i \in \{1, \dots, k\}$, then $v^+(s_{j_i + 1/2})$ is adjacent to $v^+(s_{j_{i+1}+1/2})$, which means that $S_i S_{i+1}$ is a red edge. This works analogously for blue edges of $C$, which shows the claim. Hence, by \Cref{obs:samecomps}, the $2$-factor $F(S')$ in $G'$ has the same number of components as $F(S)$ in $G$. However, since $F(S')$ is by definition just $F(U) \setminus Z$, this completes the proof.
\end{proof}

\subsubsection{Going down}
We now turn to the remaining case when we want to find a $2$-factor with less components than one that we already found.

\begin{defn}
Let $S \subseteq A$ be a disjoint union of colour-alternating cycles without neighbouring vertices. We say that a vertex $e_k \in V(A)$ \emph{separates components of $F(S)$} if the vertices $v_k$ and $v_{k+1}$ lie in different connected components of $F(S)$.
\end{defn}

\begin{obs}\label{obs:sepcomps}
If $F(S)$ has more than one connected component, then at least one vertex of $S$ separates components.
\end{obs}
\begin{proof}
Since $F(S)$ is not connected there must exist vertices $v_k,v_{k+1}$ of $H$ belonging to different components of $F(S).$ Let $e_k=v_k v_{k+1}$ so $e_k \notin E(F(S))$. Since the only edges of $H$ (that is vertices of $A$) that are not in $E(F(S))$ are vertices of $S,$ $e_k$ is the desired separating vertex. 
\end{proof}

We are now ready to construct a going-down version of $S$ giving rise to a $2$-factor with one less cycle.

\begin{defn}\label{defn:goingdown}
Let $S$ be a disjoint union of colour-alternating cycles with $V(S)=\{s_1, \dots, s_m\}$. For any $s_k \in V(S)$ we construct the $2$-edge-coloured ordered graph $D=D(S,s_k)$ as follows:
\begin{enumerate}
    \item Start with a copy of $S$ and for every vertex $s_i$ in the cycle $C\subseteq S$ that contains $s_k$, add the vertices $s_{i+1/3}$ and $s_{i+2/3}$ to $D$;
    \item if $i, j \neq k$ and if $s_i s_j$ is a red or a blue edge of $S$, then add the edges $s_{i+1/3}s_{j+1/3}$ and $s_{i+2/3}s_{j+2/3}$ of the same colour to $D$;
    \item if $s_i s_k$ is the blue edge of $S$ incident to $s_k$, then delete it and add the blue edges $s_i s_{k+1/3}$, $s_{i+1/3} s_{k+2/3}$ and $s_{i+2/3} s_k$ to $D$;
    \item if $s_i s_k$ is the red edge of $S$ incident to $s_k$, then add the red edges $s_{i+1/3} s_{k+2/3}$ and $s_{i+2/3} s_{k+1/3}$ to $D$;
    \item order the resulting graph according to the order of the indices of its vertices.
\end{enumerate}
Let $S\subseteq A$ be a disjoint union of colour alternating cycles without neighbouring vertices, so that $F(S)$ exists. We say that a $2$-edge-coloured ordered graph $D$ is a \emph{going-down version} of $S$ if there exists a vertex $s_k$ that separates components of $F(S)$ such that $D$ and $D(S,s_k)$ are isomorphic $2$-edge-coloured ordered graphs.
\end{defn}

In other words $D=D(S,s_k)$ consists of a copy of $S$ with added two copies of the cycle containing $s_k$ where the edges incident to $s_k$ and its copies are rewired in a certain way. It is easy to see that every vertex of $D$ is still incident to exactly one edge of each colour so is still a disjoint union of colour-alternating cycles. Note also that $D$ is an ordered subgraph of consistently ordered $S(3).$ If $S$ contained no double edges neither does $D$.   

\Cref{fig:goingdown} shows a going-down version $D = D(S, s_1)$ for $S$ on $\{s_1, \dots, s_4\}$ being again a colour-alternating $C_4$. Note that $F(D)$, shown in \Cref{fig:goingdownG}, contains two paths, marked as dotted and bold, that connect the two dashed parts of $F(D)$ that resemble the two disjoint cycles of $F(S),$ into a single cycle. We will show that this occurs in general. 

\begin{figure}[ht]
    \caption{A colour-alternating cycle $S$ and a going-down version of it $D(S,s_1).$}
    \includegraphics[scale = 0.8]{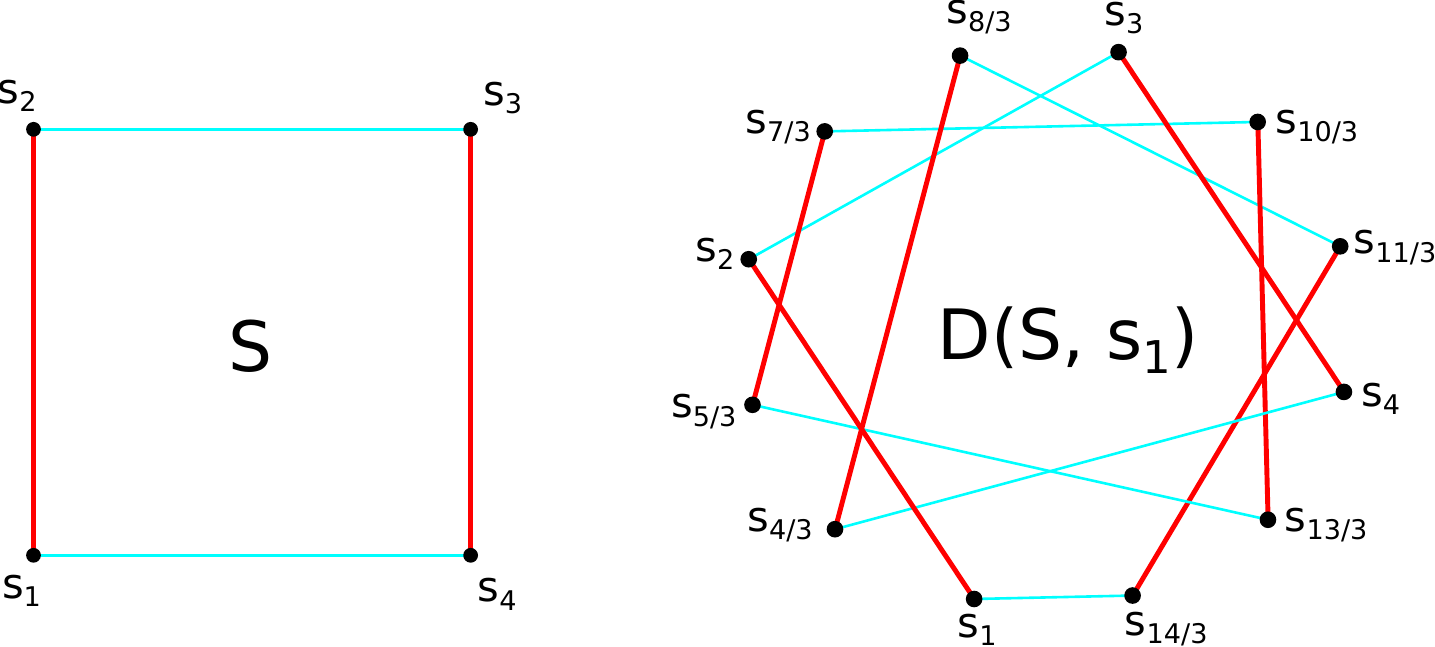}
    \label{fig:goingdown}
\end{figure}

\begin{figure}[ht]
    \caption{$2$-factors corresponding to $U$ and $D(S,s_1)$ given in \Cref{fig:goingdown}.}
    \includegraphics[scale = 0.8]{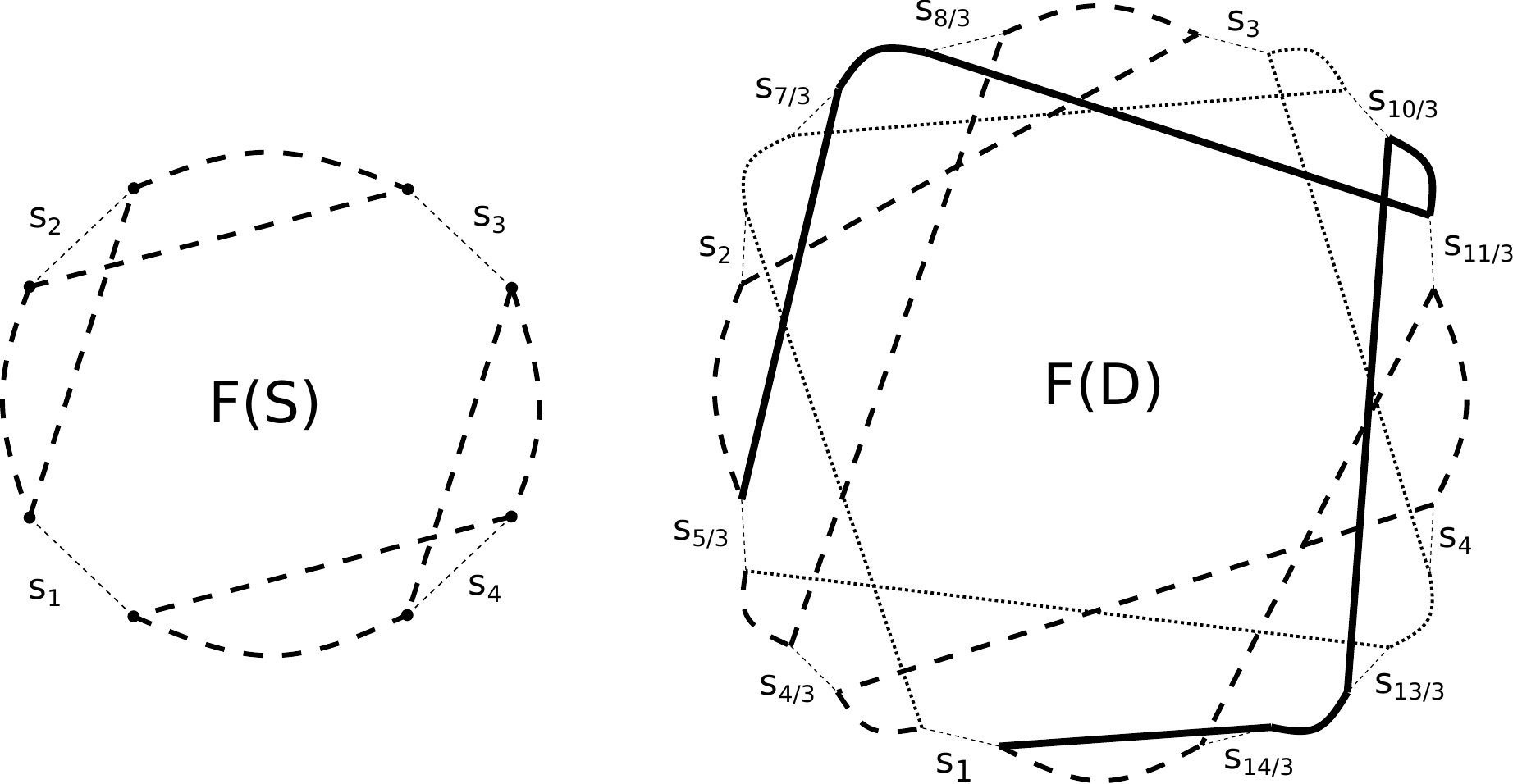}
    \label{fig:goingdownG}
\end{figure}

\begin{lem}[Going down]\label{lem:goingdown}
Let $S \subseteq A$ be a disjoint union of colour-alternating cycles without neighbouring vertices and let $D$ be an ordered subgraph of $A$ without neighbouring vertices that is a going-down version of $S$. Then the $2$-factor $F(D) \subseteq G$ consists of one cycle less than $F(S)$.
\end{lem}

\begin{proof}
For an edge $e=v_kv_{k+1}\in H$ we let $v^{+}(e)=v_{k+1}$ and $v^{-}(e)=v_k.$ 
We denote the vertices of $S$ by $s_1, \dots, s_m$   where $D=D(S,s_1)$ and $s_1$ separates components of $F(S)$. We denote the vertices of $D$ by $d_1, \dots, d_m$ and $d_{4/3}, d_{5/3}, d_{j_1+1/3}, d_{j_1 + 2/3}, \dots, d_{j_k + 2/3}$ as they appear along $H$ such that $d_1, \dots, d_m$ make a copy of $S$ in which $d_1$ corresponds to $s_1$ and $d_1,d_{j_1}, \dots, d_{j_k}$ to the cycle $C = s_1 s_{j_1} \dots s_{j_k} s_1$ of $S.$

The vertices $v^+(d_{j_i})$ and $v^-(d_{j_i + 1/3})$ as well as the vertices $v^+(d_{j_i+1/3})$ and $v^-(d_{j_i+2/3})$ in $F(D)$ are connected by paths $P_i \subseteq H$ and $Q_i \subseteq H$ respectively for all $i \in \{1, \dots, k\}$. 

If $C$ begins by a red edge then $$P:=v^+(d_1)v^+(d_{j_1})P_1v^-(d_{j_1+1/3})v^-(d_{j_2+1/3})P_2v^{+}(d_{j_2})\dots P_kv^-(d_{j_k + 1/3})v^-(d_{5/3})\in F(D),$$ where $v^+(d_1)v^+(d_{j_1})\in F(D)$ by \Cref{defn:goingdown} part 4; $v^-(d_{j_k + 1/3})v^-(d_{5/3})\in F(D)$ by part $3$ and edges between paths $P_i$ are in $F(D)$ by part 2 in the same way as in the going up case. Similarly, $$Q:=v^+(d_{5/3})v^+(d_{j_1+1/3})Q_1v^-(d_{j_1+2/3})v^-(d_{j_2+2/3})Q_2v^{+}(d_{j_2+1/3})\dots Q_kv^-(d_{j_k +2/3})v^-(d_{1})\in F(D)$$
On the other hand if $C$ begins by a blue edge then we have 
$$P:=v^-(d_{5/3})v^-(d_{j_1+1/3})P_1v^+(d_{j_1})v^+(d_{j_2})P_2\dots P_kv^+(d_{j_k})v^+(d_{1})\in F(D),$$
$$Q:=v^-(d_{1})v^-(d_{j_1+2/3})Q_1v^+(d_{j_1+1/3})v^+(d_{j_2+1/3})Q_2\dots Q_kv^+(d_{j_k+1/3})v^+(d_{5/3})\in F(D)$$
So in either case the path $P\subseteq F(D)$ contains $P_1, \dots, P_k$ and has endpoints $v^+(d_1),v^-(d_{5/3})$ while $Q \subseteq F(D)$ contains $Q_1, \dots, Q_k$ and has endpoints $v^+(d_{5/3}),v^-(d_1)$. For example in \Cref{fig:goingdownG}, the paths $P$ and $Q$ correspond to the dotted and the bold path respectively.

Our goal now is to show that $P$ and $Q$ connect two ``originally distinct'' components that are ``inherited'' from $F(S)$. Consider the graph $G'$ that is obtained from $G$ by deleting all the vertices of paths $P_i$ and $Q_i$ (equivalently all inner vertices of $P$ and $Q$) and adding the edges $S_{j_i} = v^-(d_{j_i}) v^+(d_{j_i + 2/3})$ for $i \in \{1, \dots, k\}$. Let $H'$ be the Hamilton cycle of $G'$ made of $H$ and $S_j$'s ordered according to the order of $G$.  First, we claim that the map that sends $s_1$ to $d_{4/3}$ and $s_i$ to $S_i$ if $s_i$ is part of $C \setminus \{s_1\}$ and to $d_i$ otherwise for $i \in \{2, \dots, m\}$ is an order-preserving isomorphism from $S$ onto its image $S' \subseteq A(G', H')$. Indeed, by \Cref{defn:goingdown} parts 3 and 4 for $i=1,k$ if $s_1s_{j_i}$ is red then $d_{4/3}d_{j_i+2/3}$ is a red edge of $A$ so $v^+(d_{4/3})v^+(d_{j_i+2/3})\in F(D)$ implying that $d_{4/3} S_{j_i}$ is red in $A'.$ If $s_1s_{j_i}$ is blue then $d_{4/3}d_{j_i}$ is a blue edge of $A$ so $v^-(d_{4/3})v^-(d_{j_i})\in F(D)$ implying that $d_{4/3} S_{j_i}$ is blue in $A'.$
For $i\neq k$ edge $S_{j_i} S_{j_{i+1}}$ is of the same colour as $s_{j_i}s_{j_{i+1}}$ by \Cref{defn:goingdown} part 2 and for $s_i,s_j \notin  C$ we know $d_i d_j$ has the same colour by part 1. Therefore, by \Cref{obs:samecomps}, $F(S')$ has the same number of components as $F(S)$. Since $s_1$ separates components in $S$ we know that $d_{4/3}$ separates components in $F(S')$. This means in particular that $d_1$ and $d_{5/3}$ lie in two different cycles $C_1$ and $C_2$ of $F(S')$. Now, observe that we obtain $F(D)$ from $F(S')$ by deleting $d_1$ and $d_{5/3}$ and adding the paths $P$ and $Q$. However, since $P$ connects $v^+(d_1)$ and $v^-(d_{5/3})$ and $Q$ connects $v^+(d_{5/3})$ and $v^-(d_1)$, this process joins $C_1$ and $C_2$ into one big cycle and hence, $F(D)$ has exactly one component less than $F(S)$.
\end{proof}

\subsection{Completing the proof}\label{subs:complete}

We are now ready to put all the ingredients together in order to complete our proof of \Cref{thm:main} in the way that has already been outlined throughout the previous section.

\begin{proof}[ of \Cref{thm:main}]
Let $k$ be a positive integer and $\eps$ a positive real number. Let $L=L(\eps/2),K=K(\eps/2,2^k)$ be the parameters coming from \Cref{lem:cycle-blow-up}. Let $N\ge \max(4/\eps,K).$

Now, suppose that $G$ is a Hamiltonian graph on $n \geq N$ vertices with minimum degree $\delta(G) \geq \eps n$. Let us fix a Hamilton cycle $H \subseteq G$, name the vertices of $G$ such that $H = v_1 v_2 \dots v_n v_1$ and assume that $G$ is ordered according to this labelling. Let $A = A(G,H)$ be the ordered, $2$-edge-coloured auxiliary graph corresponding to $G$ and $H$ according to \Cref{defn:aux}. We know by \Cref{deltaA} that $\delta_\nu(A)\ge \delta_\nu(G)-2 \ge \frac{\eps}{2} n.$

\Cref{lem:cycle-blow-up} shows that there is a $\mathcal{C}(2^k6^L)\subseteq A$ where $\mathcal{C}$ is a colour-alternating cycle of length at most $L$ without double-edges. \Cref{lem:ordered} allows us to find a consistently ordered $\mathcal{C}(2^k3^L)$ as an ordered subgraph of $A.$ By removing every second vertex of $\mathcal{C}(2^k3^L)$ in $A$ we obtain a consistently ordered $\mathcal{C}'=\mathcal{C}(2^{k-1}3^L)$ that is an ordered subgraph of $A$ without neighbouring vertices. For $\mathcal{C}\subseteq \mathcal{C}'$  by \Cref{correspondence} we obtain a $2$-factor $F(\mathcal{C}) \subseteq G$. Let $\ell$ be the number of cycles of $F(\mathcal{C}).$ By \Cref{obs:no.cycles}, we know that $1 \leq \ell \leq L$.

Let us first assume that $k > \ell$. We find a sequence $S_0, S_1, \dots, S_{k-\ell}$ defined as follows: let $S_0 = \mathcal{C};$ given $S_{i-1}$ let $\mathcal{C}_{i-1}$ be an arbitrary cycle of $S_{i-1}$ and let $S_{i} = U(S_{i-1}, \mathcal{C}_{i-1})$. By construction, $S_i$ is again a disjoint union of colour-alternating cycles, without double edges, and is an ordered subgraph of $\mathcal{C}(2^i) \subseteq \mathcal{C}'$ (since by construction $S_i\subseteq S_{i-1}(2)$). Therefore, for all $i\le k-\ell$ there is an order-preserving embedding of $S_i$ into $A$ without neighbouring vertices. So, by \Cref{correspondence} and \Cref{lem:goingup} we deduce that $F(S_i)$ has one more cycle than $F(S_{i-1})$. In particular, the $2$-factor $F(S_{k-\ell}) \subseteq G$ consists of exactly $k$ components. 

Let us now assume that $k < \ell$. Here, we find a sequence $S_0, S_1, \dots, S_{\ell - k}$ of disjoint unions of colour-alternating cycles that are ordered subgraphs of $A$ without neighbouring vertices such that $F(S_i)$ consists of $\ell-i$ cycles. Let $S_0 = \mathcal{C},$ and assume we are given $S_{i-1}$ for $i\le \ell-k$ with $F(S_{i-1})$ having $\ell-i+1 \ge k+1 \ge 2$ cycles. This means that $S_{i-1}$ has a vertex $v_{i-1}$ that separates components of $F(S_{i-1})$ by \Cref{obs:sepcomps}. We let $S_i = D(S_{i-1}, v_{i-1})$, which is a disjoint union of colour-alternating cycles, without double edges, and is an ordered subgraph of a consistently ordered $\mathcal{C}(3^i)$ (since by construction $S_i \subseteq S_{i-1}(3)$). Note that $\ell - k \leq L$ by \Cref{obs:no.cycles} and hence, $ \mathcal{C}(3^i) \subseteq \mathcal{C}(3^{\ell - k}) \subseteq \mathcal{C}'$ so we can find a copy of $S_i$ into $A$ without having neighbouring vertices. By \Cref{lem:goingdown}, $F(S_i)$ has one less cycle than $F(S_{i-1})$, so exactly $\ell - i$ cycles. In particular, $F(S_{\ell - k})$ is a $2$-factor in $G$ with $k$ cycles, which concludes the proof.
\end{proof}

\section{Concluding remarks and open problems}
In this paper we show that in a Hamiltonian graph the minimum degree condition needed to guarantee any $2$-factor with $k$-cycles is sublinear in the number of vertices. The best lower bound is still only a constant. 
In the case of a $2$-factor with two components, the best bounds are given by Faudree et al. \cite{Faudree05} who construct minimum degree $4$ Hamiltonian graphs without a $2$-factor with $2$ components. In the case of $2$-factors with $k$ components, no constructions have been given previously, but it is easy to see that a minimum degree of at least $k+2$ is necessary:
\begin{prop}
There are arbitrarily large Hamiltonian graphs with minimum degree $k+1$ which do not have a $2$-factor with $k$ components.
\end{prop}
\begin{proof}
Let $G$ consist of a cycle $\mathcal{C}$ of length $n-k+1$ and an independent set $U$ of size $k-1$ with all the edges between $\mathcal{C}$ and $U$ added. It's easy to see that for $n\ge 2k$,  $G$ is Hamiltonian and has minimum degree $k+1$. However $G$ does not have a $2$-factor with $k$ components (e.g.\ because every cycle in a $2$-factor of $G$ must use at least one vertex in $U$).
\end{proof}
For fixed $k$, we do not know of any Hamiltonian graphs with non-constant minimum degree which do not have a $2$-factor with $k$ components. This indicates that the necessary minimum degree in \Cref{conj:main} may in fact be much smaller, perhaps even a constant (depending on $k$). A step in this direction was made by Pfender \cite{Pfender04} who showed that in the $k=2$ case, a Hamiltonian graph $G$ with minimum degree of $7$ contains a $2$-factor with $2$ cycles in a very special case when $G$ is claw-free. 

If one takes greater care with various parameters in \Cref{sec:prelim} one can show that a minimum degree of $\frac{Cn}{\sqrt[4]{\log \log n / (\log \log \log n)^2}}$ suffices for finding an ordered blow-up of a short cycle so in particular this minimum degree is enough to find $2$-factors consisting of a fixed number of cycles. We believe that it would be messy but not too hard to improve this a little bit further, but to reduce the minimum degree condition to $n^{1-\eps}$ would require some new ideas. On the other hand we do believe that our approach of finding alternating cycles in the auxiliary graph could still be useful in this case, but one needs to either find a better way of finding ordered blow-ups of short cycles or obtain a better understanding of how the number of cycles in $F(S)$ depends on the order and structure of a disjoint union of colour-alternating cycles $S.$ Another possibility is to augment the auxiliary graph in order to include edges that connect the front/back to the back/front vertex of two edges of the Hamilton cycle, which would allow us to obtain a $1$-to-$1$-correspondence between $2$-factors of $G$ and suitable structures in this new auxiliary graph. 

Another way of saying that a graph is Hamiltonian is that it has a $2$-factor consisting of a single cycle. A possibly interesting further question which arises is whether knowing that $G$ contains a $2$-factor consisting of $\ell$ cycles already allows the minimum degree condition needed for having a $2$-factor with $k>\ell$ cycles to be weakened. 

\textbf{Acknowledgements.} We are extremely grateful to the anonymous referees for their careful reading of the paper and many useful suggestions.


\begin{thebibliography}{10}

\bibitem{brandt97}
{\sc Brandt, S., Chen, G., Faudree, R., Gould, R.~J., and Lesniak, L.}
\newblock Degree conditions for {$2$}-factors.
\newblock {\em J. Graph Theory 24}, 2 (1997), 165--173.

\bibitem{DeBiasio14}
{\sc DeBiasio, L., Ferrara, M., and Morris, T.}
\newblock Improved degree conditions for 2-factors with {$k$} cycles in
  {H}amiltonian graphs.
\newblock {\em Discrete Math. 320\/} (2014), 51--54.

\bibitem{Dirac52}
{\sc Dirac, G.~A.}
\newblock Some theorems on abstract graphs.
\newblock {\em Proc. London Math. Soc. (3) 2\/} (1952), 69--81.

\bibitem{kst}
{\sc Erd\H{o}s, P.}
\newblock On extremal problems of graphs and generalized graphs.
\newblock {\em Israel J. Math. 2\/} (1964), 183--190.

\bibitem{Faudree05}
{\sc Faudree, R.~J., Gould, R.~J., Jacobson, M.~S., Lesniak, L., and Saito, A.}
\newblock A note on 2-factors with two components.
\newblock {\em Discrete Math. 300}, 1-3 (2005), 218--224.

\bibitem{gouldsurvey}
{\sc Gould, R.~J.}
\newblock Updating the {H}amiltonian problem---a survey.
\newblock {\em J. Graph Theory 15}, 2 (1991), 121--157.

\bibitem{gyori12}
{\sc Gyori, E., and Li, H.}
\newblock 2-factors in {H}amiltonian graphs.
\newblock {\em In preparation\/}.

\bibitem{lisurvey}
{\sc Li, H.}
\newblock Generalizations of {D}irac's theorem in {H}amiltonian graph
  theory---a survey.
\newblock {\em Discrete Math. 313}, 19 (2013), 2034--2053.

\bibitem{Pfender04}
{\sc Pfender, F.}
\newblock 2-factors in {H}amiltonian graphs.
\newblock {\em Ars Combin. 72\/} (2004), 287--293.

\bibitem{Sarkozy08}
{\sc S\'{a}rk\"{o}zy, G.~N.}
\newblock On 2-factors with k components.
\newblock {\em Discrete Math. 308}, 10 (May 2008), 1962--1972.

\bibitem{bennysurvey}
{\sc Sudakov, B.}
\newblock Robustness of graph properties.
\newblock In {\em Surveys in combinatorics 2017}, vol.~440 of {\em London Math.
  Soc. Lecture Note Ser.} Cambridge Univ. Press, Cambridge, 2017, pp.~372--408.

\end{thebibliography}
\end{document}